\documentclass[a4paper,12pt]{article}
\usepackage{amsfonts}
\usepackage{}
\usepackage{mathrsfs}
\usepackage{algorithm, algpseudocode}
\usepackage{amsmath,amssymb,amsthm,latexsym,amsfonts}
\usepackage{pstricks}
\usepackage{enumerate}

\usepackage{graphicx}
\usepackage{color}
\usepackage{ifpdf}

\usepackage{caption}

\begin{document}\baselineskip 0.8cm

\newtheorem{lem}{Lemma}[section]
\newtheorem{thm}[lem]{Theorem}
\newtheorem{cor}[lem]{Corollary}
\newtheorem{exa}[lem]{Example}
\newtheorem{con}[lem]{Conjecture}
\newtheorem{rem}[lem]{Remark}
\newtheorem{obs}[lem]{Observation}
\newtheorem{definition}[lem]{Definition}
\newtheorem{pro}[lem]{Proposition}
\theoremstyle{plain}
\newcommand{\D}{\displaystyle}
\newcommand{\DF}[2]{\D\frac{#1}{#2}}

\renewcommand{\figurename}{{\bf Fig}}
\captionsetup{labelfont=bf}

\title{Hardness result for the total rainbow $k$-connection of graphs\footnote{Supported by NSFC No.11371205 and 11531011.}}
\author{\small Wenjing Li, Xueliang Li, Di Wu\\
      {\it\small $^1$ Center for Combinatorics and LPMC-TJKLC}\\
      {\it\small Nankai University, Tianjin 300071, P.R. China}\\
       {\it\small liwenjing610@mail.nankai.edu.cn; lxl@nankai.edu.cn; wudiol@mail.nankai.edu.cn}}
\date{}
\maketitle

\begin{abstract}
A path in a total-colored graph is called \emph{total rainbow} if its edges and internal vertices
have distinct colors.
For an $\ell$-connected graph $G$ and an integer $k$ with $1\leq k \leq\ell$,
the \emph{total rainbow $k$-connection number} of $G$,
denoted by $trc_k(G)$,
is the minimum number of colors used in a total coloring of $G$
to make $G$ \emph{total rainbow $k$-connected}, that is,
any two vertices of $G$ are connected by $k$ internally vertex-disjoint total rainbow paths.
In this paper, we study the computational complexity of 
total rainbow $k$-connection number of graphs.
We show that it is NP-complete to decide whether $trc_k(G)=3$.
\\[2mm]


\end{abstract}

\section{Introduction}
All graphs considered in this paper are simple, finite, undirected and connected.
We follow the terminology and notation of
Bondy and Murty\cite{Bondy} for those not defined here.
A set of internally vertex-disjoint paths are called \emph{disjoint}.
Let $G$ be a nontrivial connected graph
with an \emph{edge-coloring c} $:E(G)\rightarrow \{0,1,\dots,t\}$, $t\in \mathbb{N}$,
where adjacent edges may be colored the same.
A path in $G$ is called \emph{a rainbow path}
if no two edges of the path are colored the same.
The graph $G$ is called \emph{rainbow connected}
if for any two vertices of $G$,
there is a rainbow path connecting them.
The \emph{rainbow connection number} of $G$,
denoted by $rc(G)$,
is defined as the minimum number of colors that are needed
to make $G$ rainbow connected.
If $G$ is an $\ell$-connected graph with $\ell\geq1$,
then for any integer $1\leq k\leq \ell$,
$G$ is called \emph{rainbow $k$-connected}
if any two vertices of $G$ are connected by $k$ disjoint rainbow paths.
The \emph{rainbow $k$-connection number} of $G$,
denoted by $rc_k(G)$,
is the minimum number of colors that are required
to make $G$ rainbow $k$-connected.
The concepts of rainbow connection and rainbow $k$-connection of graphs
were introduced by Chartrand et al. in \cite{Char2,Char1},
and have been well-studied since then.
For further details, we refer the readers to the book\cite{Li}.

Let $G$ be a nontrivial connected graph
with a \emph{vertex-coloring c} $:V(G)\rightarrow \{0,1,\dots,t\}$, $t\in \mathbb{N}$,
where adjacent vertices may be colored the same.
A path in $G$ is called \emph{a vertex-rainbow path}
if no interval vertices of the path are colored the same.
The graph $G$ is \emph{rainbow vertex-connected}
if for any two vertices of $G$,
there is a vertex-rainbow path connecting them.
The \emph{rainbow vertex-connection number} of $G$,
denoted by $rvc(G)$,
is the minimum number of colors used in a vertex-coloring of $G$
to make $G$ rainbow vertex-connected.
If $G$ is an $\ell$-connected graph with $\ell\geq1$,
then for any integer $1\leq k\leq \ell$,
the graph $G$ is \emph{rainbow vertex $k$-connected}
if any two vertices of $G$ are connected by $k$ disjoint vertex-rainbow paths.
The \emph{rainbow vertex $k$-connection number} of $G$,
denoted by $rvc_k(G)$,
is the minimum number of colors that are required
to make $G$ rainbow vertex k-connected.
These concepts of rainbow vertex connection
and rainbow vertex $k$-connection of graphs
were proposed by Krivelevich and Yuster\cite{Yuster} and Liu et al.\cite{H. L1}, respectively.

Liu et al.\cite{H. L2} introduced the analogous concepts of total rainbow
$k$-connection of graphs.
Let $G$ be a nontrivial $\ell$-connected graph
with a \emph{total-coloring c} $:E(G)\cup V(G)
\rightarrow \{0,1,\dots,t\}$, $t\in \mathbb{N}$, where $\ell\geq1$.
A path in $G$ is called \emph{a total-rainbow path}
if its edges and interval vertices have distinct colors.
For any integer $1\leq k\leq\ell$,
the graph $G$ is called \emph{total rainbow $k$-connected}
if any two vertices of $G$ are connected by $k$ disjoint total-rainbow paths.
The \emph{total rainbow $k$-connection number} of $G$,
denoted by $trc_k(G)$,
is the minimum number of colors that are needed
to make $G$ total rainbow $k$-connected.
When $k=1$, we simply write $trc(G)$, just like
$rc(G)$ and $rvc(G)$.
From Liu et al.\cite{H. L2}, we have that $trc(G)=1$ if and only if $G$ is a complete graph,
and $trc(G)\geq 3$ if $G$ is not complete.
If $G$ is an $\ell$-connected graph with $\ell\geq1$,
then $trc_k(G)\geq 3$ if $2\leq k\leq \ell$,
and $trc_k(G)\geq 2diam(G)-1$ for $1\leq k\leq \ell$, where $diam(G)$ denotes the diameter of $G$.
In relation to $rc_k(G)$ and $rvc_k(G)$, they have $trc_k(G)\geq max(rc_k(G),rvc_k(G))$.
Also, if $rc_k(G)=2$, then $trc_k(G)=3$.
If $rvc_k(G)\geq 2$, then $trc_k(G)\geq 5$.

The computational complexity of the rainbow connectivity and vertex-connectivity has been
attracted much attention. In \cite{Chakraborty}, Chakraborty et al. proved
that deciding whether $rc(G)=2$ is NP-Complete. Analogously, Chen et al.\cite{Chen} showed that it is NP-complete
to decide whether $rvc(G)=2$. Motivated by \cite{Chakraborty,Chen}, we consider the computational complexity of
computing the total rainbow $k$-connectivity $trc_k(G)$ of a graph $G$. For $k=1$, Chen et al. recently gave reductions to prove that
it is NP-complete to decide whether $trc(G)=3$ in \cite{Lily chen}. In this paper, we prove that for any fixed $k\geq 1$ it is NP-complete to
decide whether $trc_k(G)=3$. The reduction of our proof is different from that in \cite{Lily chen}.

\section{Main results}

In the following, we will show that deciding whether $trc_k(G)=3$ is NP-complete for fixed $k\geq 1$.

\begin{thm}\label{thm1}
Given a graph $G$, deciding whether $trc_k(G)=3$ is NP-Complete for fixed $k\geq 1$.
\end{thm}

We first define the following three problems.

\noindent{\bf Problem 1.} The total rainbow connection number $3$.

\noindent Given: Graph $G=(V,E)$.

\noindent Decide: Whether there is a total coloring of $G$ with $3$ colors
such that all the pairs $\{u,v\}\in (V\times V)$ are total rainbow $k$-connected?

\noindent{\bf Problem 2.} The subset total rainbow $k$-connection number $3$.

\noindent Given: Graph $G=(V,E)$ and a set of pairs $P\subseteq (V\times V)$,
where $P$ contains nonadjacent vertex pairs.

\noindent Decide: Whether there is a total-coloring of $G$ with $3$ colors
such that all the pairs $\{u,v\}\in P$ are total rainbow $k$-connected?

\noindent{\bf Problem 3.} The subset partial edge-coloring.

\noindent Given: Graph $G=(V,E)$ with a set of pairs $Q\subseteq V\times V$
where $Q$ contains nonadjacent vertex pairs,
and a partial $2$-edge-coloring $\hat{\chi}$ for $\hat{E}\subset E$.

\noindent Decide: Whether $\hat{\chi}$ can be extended to a $3$-total-coloring
$\chi$ of $G$ that makes all the pairs in $Q$ total rainbow $k$-connected and
$\chi(e)\notin \{\chi(u),\chi(v)\}$ for all $e=uv\in \hat{E}$?

In the following, we first reduce Problem 2 to Problem 1,
and then reduce Problem 3 to Problem 2.
Finally, Theorem \ref{thm1} is completed by reducing $3$-SAT to Problem 3.

Before proving Theorem \ref{thm1}, we need an useful result shown in \cite{Char2}.

\begin{lem}\cite{Char2}\label{lem1}
For every $k\geq 2$, $rc_k(K_{(k+1)^2})=2$.
Furtherly, the following 2-edge coloring can make $G$ rainbow $k$-connected.
Let $G_1,G_2,\ldots,G_{k+1}$ be mutually vertex-disjoint graphs,
where $V(G_i)=V_i$, such that $G_i=K_{k+1}$ for $1\leq i\leq k+1$.
Let $V_i=\{v_{i,1},v_{i,2},\ldots,v_{i,k+1}\}$ for $1\leq i\leq k+1$.
Let $G$ be the join of the graphs $G_1,G_2,\ldots,G_{k+1}$.
Thus $G=K_{(k+1)^2}$ and $V(G)=\cup^{k+1}_{i=1}V_i$.
We assign the edge $uv$ of $G$ the color $0$ if either $uv\in E(G_i)$
for some $i(1\leq i\leq k+1)$ or if $uv=v_{i,l}v_{j,l}$ for some $i,j,l$ with $1\leq i,j,l\leq k+1$
and $i\neq j$.
All other edges of $G$ are assigned the color $1$.

For $k=1$, since $rc_1(K_{(k+1)^2})$=1, the above coloring surly makes $G$ rainbow $1$-connected.
\end{lem}
Note that from the above coloring, for every vertex $v\in V(G)$,
we have $d(v)=k^2+2k$, $2k$ edges incident with $v$ colored with $0$,
and $k^2$ edges incident with $v$ colored with $1$.

\begin{lem}\label{lem2}
Problem 2 $\preceq$ Problem 1.
\end{lem}
\begin{proof}
Given a graph $G=(V,E)$ and a set of pairs $P\subseteq V\times V$
where $P$ contains nonadjacent vertex pairs,
we construct a graph $G'=(V',E')$ as follows.
For every vertex $v\in V$, we introduce a new vertex set $V_v=\{x_{(v,1)},x_{(v,2)},\ldots,x_{(v,(k+1)^2)}\}$,
and for every pair $\{u,v\}\in (V\times V)\setminus P$,
we introduce a new vertex set $V_{(u,v)}=\{x_{(u,v,1)},x_{(u,v,2)},\ldots,x_{(u,v,(k+1)^2)}\}$.
We set
\begin{align*}
V'=V\cup \{V_v: v\in V\}\cup \Big\{V_{(u,v)}: \{u,v\}\in (V\times V)\setminus P\Big\}
\end{align*}
and
\begin{align*}
E'= E & \cup \{vx_{(v,i)}: v\in V, x_{(v,i)}\in V_{v}\}\\
&\cup \Big\{\{ux_{(u,v,i)},vx_{(u,v,i)}\}: \{u,v\}\in (V\times V)\setminus P, x_{(u,v,i)}\in V_{(u,v)}\Big\}\\
& \cup \{xx': x,x'\in V'\setminus V\}.\\
\end{align*}
It remains to verify that $G'$ is $3$-total rainbow $k$-connected if and only if there is a total-coloring of $G$
with $3$ colors such that all the pairs $\{u,v\}\in P$ are total rainbow $k$-connected.
In one direction,
suppose that $G'$ is $3$-total rainbow $k$-connected.
Notice that when $G$ is considered as a subgraph of $G'$, no pair of vertices of $G$ that appear in $P$
has a path of length two in $G'$ that is not fully contained in $G$.
Then with this coloring, all the pairs $\{u,v\}\in P$ are total rainbow $k$-connected in $G$.

In the other direction, suppose that $\chi: V\cup E\rightarrow\{0,1,2\}$ is a total-coloring of $G$
that makes all the pairs in $P$ total rainbow $k$-connected.
We now extend it to a total rainbow $k$-connection coloring $\chi': V'\cup E'\rightarrow \{0,1,2\}$,
$\chi'(x)=2$ for all $x\in V'\setminus V$;
$\chi'(v,x_{(v,i)})=1$ for all $v\in V$ and $x_{(v,i)}\in V_v$;
$\chi'(u,x_{(u,v,i)})=0,\chi'(v,x_{(u,v,i)})=1$ for all $\{u,v\}\in (V\times V)\setminus P$
and all $x_{(u,v,i)}\in V_{(u,v)}$.
The edges in $G'[V_v]$ or $G'[V_{(u,v)}]$ are colored with $\{0,1\}$ as Lemma \ref{lem1}
for all $v\in V$ and all $\{u,v\}\in (V\times V)\setminus P$.
Finally, the remaining uncolored edges are colored with $0$.
Now we show that $G'$ is total rainbow $k$-connected under this coloring.
For $\{u,v\}\in P$, the $k$ disjoint total rainbow paths in $G$ connecting $u$ and $v$ are also
$k$-disjoint total rainbow paths in $G'$.
For $\{u,v\}\in (V\times V)\setminus P$, $\{ux_{(u,v,1)}v,ux_{(u,v,2)}v,\ldots,ux_{(u,v,k)}v\}$
are $k$ disjoint total rainbow paths.
For $u\in V,v\in V'\setminus V$,
if $v\notin V_u$, then $\{ux_{(u,1)}v,ux_{(u,2)}v,\ldots,ux_{(u,k)}v\}$
are $k$ disjoint total rainbow paths;
if $v\in V_u$, from Lemma \ref{lem1},
we have $2k>k$ edges incident with $v$ are colored with $0$ in $G'[V_u]$.
Suppose that $\{v_1,v_2,\ldots,v_k\}$ are $k$ vertices adjacent to $v$ by these edges colored with $0$,
then $\{uv_1v,uv_2v,\ldots,uv_kv\}$ are $k$ disjoint total rainbow paths.
For $\{x,x'\}\in (V_u\times V_u)$ or $(V_{(u,v)}\times V_{(u,v)})$,
by Lemma \ref{lem1}, there are $k$ disjoint total rainbow paths in $G'[V_u]$ or $G'[V_{(u,v)}]$
connecting $u$ and $v$.
For the remaining pairs $\{x,x'\}$,
suppose w.l.o.g that $x\in V_u$ and $x'\in V_v(u\neq v)$.
By Lemma \ref{lem1}, we have $k^2>k$ edges incident with $v$ are colored with $1$ in $G'[V_v]$.
Suppose that $\{v'_1,v'_2,\ldots,v'_k\}$ are $k$ vertices adjacent with $v$ by these edges colored with $1$,
then $\{uv'_1v,uv'_2v,\ldots,uv'_kv\}$ are $k$ disjoint total rainbow paths.
Hence ${\chi'}$ is indeed a total $k$-rainbow coloring of $G'$.
\end{proof}

\begin{lem}\label{lem3}
Problem $3\preceq$ Problem $2$.
\end{lem}
\begin{proof}
Since the identity of the colors does not matter,
it is more convenient that instead of a partial $2$-edge coloring $\hat{\chi}$
we consider the corresponding partition $\pi_{\hat{\chi}}=(\hat{E_1},\hat{E_2})$.
For the sake of convenience, let $e=e^1e^2$ for $e\in (\hat{E_1}\cup \hat{E_2})$.
Note that the ends of $e$ may be labeled by different signs for $e\in (\hat{E_1}\cup \hat{E_2})$.
Given such a partial $2$-edge coloring $\hat{\chi}$ and a set of pairs $Q\subseteq (V\times V)$
where $Q$ contains nonadjacent vertex pairs.
Now we construct a graph $G'=(V',E')$ and define a set of pairs $P\subseteq (V'\times V')$ as follows.
We first add the vertices
\begin{align*}
\{c,b_1,b_2\}\cup \Big\{\{c^j_e,d^j_e,f^j_e\}: j\in\{1,2\}, e\in(\hat{E_1}\cup \hat{E_2})\Big\}
\end{align*}
and add the edges
\begin{align*}
\Big\{b_1c,b_2c\Big\}\cup \Big\{cc^j_e: j\in\{1,2\}, e\in (\hat{E_1}\cup \hat{E_2})\Big\}
\cup \Big\{ \{c^j_ef^j_e,c^j_ee^j,d^j_ee^j\}: e\in (\hat{E_1}\cup \hat{E_2})\Big\}.
\end{align*}
Now we define the set of pairs $P$.
\begin{align*}
P= & Q\cup \{b_1,b_2\}\cup \Big\{\{b_i,c^j_e\}: e\in \hat{E_i}, i,j\in \{1,2\}\Big\}\\
& \cup \Big\{\{f^j_e,c\},\{f^j_e,e^j\},\{d^j_e,c^j_e\},\{d^j_e,e^{(3-j)}\}: j\in \{1,2\}, e\in(\hat{E_1}\cup \hat{E_2})\Big\}.
\end{align*}
Then we secondly add the new vertices
\begin{align*}
\Big\{\{g_{(u,v,2)},g_{(u,v,3)},\ldots,g_{(u,v,k)}\}: \{u,v\}\in P\setminus Q\Big\}
\end{align*}
and add the new edges
\begin{align*}
\Big\{\{ug_{(u,v,2)}v,ug_{(u,v,3)}v,\ldots,ug_{(u,v,k)}v\}: \{u,v\}\in P\setminus Q\Big\}.
\end{align*}
On one hand, if there is a $3$-total-coloring of $\chi$ of $G$
that makes all the pairs in $Q$ total rainbow $k$-connected which extends $\pi_{\hat{\chi}}=(\hat{E_1},\hat{E_2})$
and $\chi(e)\notin \{\chi(e^1),\chi(e^2)\}$ for all $e=e^1e^2\in \hat{E}$,
then we give a total-coloring $\chi'$ of $G'$ as follows.
Suppose w.l.o.g that $\hat{E_1}$ are colored with $0$, and $\hat{E_2}$ are colored with $1$.
$\chi'(v)=\chi(v)$, and $\chi'(e)=\chi(e)$ for all $v\in V,e\in E$;
$\chi'(v)=2$ for all $v\in V'\setminus V$;
$\chi'(b_1c)=1$, and $\chi'(b_2c)=0$;
$\chi'(c^j_ee^j)=\chi'(c^j_ec)=0$, and $\chi'(d^j_ee^j)=\{1,2\}\setminus \chi(e^j)$ for all $e\in \hat{E_1}$;
$\chi'(c^j_ee^j)=\chi'(c^j_ec)=1$, and $\chi'(d^j_ee^j)=\{0,2\}\setminus \chi(e^j)$ for all $e\in \hat{E_2}$;
$\chi'(ug_{(u,v,t)})=0$, and $\chi'(g_{(u,v,t)}v)=1$ for all $2\leq t\leq k$ and all $\{u,v\}\in P\setminus Q$.
One can verify that this coloring indeed makes all the pairs in $P$ total rainbow $k$-connected.

On the other hand, any $3$-total-coloring of $G'$ that makes all the pairs in $P$ total rainbow $k$-connected
indeed makes all the pairs in $Q$ total rainbow $k$-connected in $G$,
because $G'$ contains no path of length $2$ between any pair in $Q$ that is not contained in $G$.
Note that there exactly exist $k$ disjoint total rainbow paths between any pair in $P\setminus Q$.
For any $e\in \hat{E_i}$, $i\in\{1,2\}$, from the set of pairs
$\Big\{\{b_1,b_2\}$, $\{b_i,c^j_e\}$, $\{f^j_e,c\}$, $\{f^j_e,e^j\}$, $\{d^j_e,c^j_e\}$, $\{d^j_e,e^{(3-j)}\}$: $j\in \{1,2\}\Big\}$,
we have $\chi'(b_1c)\neq \chi'(b_2c)$, $\chi'(e)=\chi'(c^j_ee^j)=\chi'(c^j_ec)=\chi'(b_{(3-i)}c)$
and $\chi'(e)\notin \{\chi'(e^1),\chi'(e^2)\}$ for $j\in \{1,2\}$.
Hence the coloring $\chi'$ of $G'$ not only provides a $3$-total-coloring $\chi$ of $G$
that makes all the pairs in $Q$ are total rainbow $k$-connected, but it also make sure that
$\chi$ extends the original partial coloring $\pi_{\hat{\chi}}=(\hat{E_1},\hat{E_2})$ and
$\chi(e)\notin \{\chi(e^1),\chi(e^2)\}$ for all $e=e^1e^2\in \hat{E}$.
\end{proof}
\begin{lem}\label{lem4}
$3$-SAT$\preceq$ Problem 3.
\end{lem}
\begin{proof}
Given a 3CNF formula $\phi=\bigwedge^m_{i=1}c_i$ over variables $\{x_1,x_2,\ldots,x_n\}$,
we construct a graph $G=(V,E)$, a partial $2$-edge coloring suppose w.l.o.g that $\hat{\chi}:\hat{E}\rightarrow\{0,1\}$,
and a set of pairs $Q\subseteq (V\times V)$ where $Q$ contains nonadjacent vertex pairs
such that there is an extension $\chi$ of $\hat{\chi}$
that makes all the pairs in $Q$ total rainbow $k$-connected and $\chi(e)\notin \{\chi(u),\chi(v)\}$ for all $e=uv\in \hat{E}$
if and only if $\phi$ is satisfiable.
We define $G$ as follows:
\begin{align*}
V(G)=\{c_t: t\in [m]\}\cup \{c^j_t,t\in [m], 2\leq j\leq k \}\cup  \{x_i: i\in [n]\}\cup \{s\}
\end{align*}
and
\begin{align*}
E(G)=\{c_tx_i: x_i\in c_t~\text{in}~\phi\}\cup \{sx_i: i\in [n]\}
\cup \Big\{\{sc^j_t,c^j_tc_t\}:t\in [m], 2\leq j\leq k\Big\}.
\end{align*}
Now we define the set of pairs $Q$ as follows:
\begin{align*}
Q=\Big\{\{s,c_t\}: t\in [m]\Big\}.
\end{align*}
Finally we define the partial $2$-edge coloring $\hat{\chi}$ as follows:
\begin{align*}
\hat{\chi}(c_tx_i)=
\begin{cases}
0 & \text{if $x_i$ is positive in $c_t$},\\
1 & \text{if $x_i$ is negative in $c_t$}.
\end{cases}
\end{align*}
On one hand, if $\phi$ is satisfiable with a truth assignment over $\{x_1,x_2,\ldots,x_n\}$,
we extend $\hat{\chi}$ to $\chi$ as follows:
$\chi(v)=2$ for all $v\in V$;
$\chi(sc^j_t)=0$, and $\chi(c^j_tc_t)=1$ for all $t\in [m]$ and all $2\leq j\leq k$;
$\chi(sx_i)=x_i$ for all $i\in [n]$.
One can verify $\chi$ is as desired.
On the other hand,
suppose that $\chi$ is as desired as above.
Note that for any $c_t$, there must exist a total rainbow path $sx_ic_t$ by some vertex $x_i$.
Set such $x_i=\{\chi(sx_i),\chi(x_i)\}\cap \{0,1\}$ which can make $c_t$ true.
One can verify $\phi$ is satisfiable.
\end{proof}

\end{document}